\documentclass[12pt]{amsart}          
\usepackage[a4paper]{geometry}	   
\usepackage{amsfonts,amsmath,latexsym,amssymb} 
\usepackage{cite}
\usepackage[T2A]{fontenc}
\usepackage[utf8]{inputenc}
\usepackage[english]{babel}
\usepackage{appendix}

\newtheorem{theorem}{Theorem}
\newtheorem{lemma}{Lemma}
\newtheorem{remark}{Remark}
\newtheorem{definition}{Definition}
\newtheorem{corollary}{Corollary}

\DeclareMathOperator*{\esssup}{{\rm ess\,sup}}
\newcommand{\Rd}{{\mathbb R}^d}
\newcommand{\RR}{\mathbb R}
\newcommand{\setsClass}{\mathbb K}
\newcommand{\conesClass}{\mathcal{C}}
\newcommand{\polarK}{K^\circ}
\newcommand{\weightsClass}{\mathcal{W}}
\newcommand{\weightedLp}{\mathcal{L}}
\newcommand{\mes}{{\rm meas\,}}
\newcommand{\KC}{K\cap C}

\allowdisplaybreaks
\begin{document}

\title[On approximation of hypersingular integral operators...]{On approximation of hypersingular integral operators by bounded ones}
\author[V.~Babenko]{Vladyslav Babenko}
\address{Department of Mathematical Analysis and Theory of Functions, Oles Honchar Dnipro National University, Dnipro, Ukraine}

\author[O.~Kovalenko]{Oleg Kovalenko}
\address{Department of Mathematical Analysis and Theory of Functions, Oles Honchar Dnipro National University, Dnipro, Ukraine}

\author[N.~Parfinovych]{Nataliia Parfinovych}
\address{Department of Mathematical Analysis and Theory of Functions, Oles Honchar Dnipro National University, Dnipro, Ukraine}

\keywords{Hypersingular integral operator, Stechkin's problem, Landau-Kolmogorov type inequality}

\subjclass[2010]{26D10, 41A17, 41A44, 41A55}

\begin{abstract}
We solve the Stechkin problem about approximation of generally speaking unbounded hypersingular integral operators by bounded ones. As a part of the proof, we also solve several related and interesting on their own problems. In particular, 
we obtain sharp Landau-Kolmogorov type inequalities in both additive and multiplicative forms for hypersingular integral operators and prove a sharp Ostrowski type inequality for multivatiate Sobolev classes. We also give some applications of the obtained results, in particular study the modulus of continuity of the hypersingular integral operators, and solve the problem of optimal recovery of the value of a hypersingular integral operator based on the argument known with an error.
\end{abstract}

\maketitle

\section{Introduction}
Let $X$ and $Y$ be Banach spaces, $\mathcal{L}(X,Y)$ be the space of linear bounded operators $S\colon X\to Y$, and let $A\colon X \to Y$ be an operator (not necessarily linear) with the domain $D_A \subset X$. Let also $Q\subset D_A$ be some class of elements.

The problem to find the quantity 
\begin{equation} \label{uklonenie}
    U(A,S; Q)=\sup_{x\in Q}\|Ax-Sx\|_Y
\end{equation}
plays an important role in Approximation Theory and Numerical Analysis. Such problem occurs, for example, in the studies of the approximation of functional classes by linear methods, in estimates of the error of  various numeric methods, in particular in formulae of numeric integration, and other investigations.

Quantity~\eqref{uklonenie} is connected to a series of optimization problems. For example, treating an operator $A$ as given, we may want to find an operator $S$ from some family of operators that minimizes the right-hand side of~\eqref{uklonenie}. A classical example of this kind is the problem to find the best linear method of approximation of a given functional class by a given approximating set.

In this article we are interested in approximation of a generally speaking unbounded  hypersingular integral operator
by linear operators with norm not exceeding a given number $N$. The problem is to find the quantity
$$  
E_N(A,Q)=\inf\left\{ U(A,S;Q)\colon S\in{\mathcal L}(X,Y), \|S\|\le N\right\}.
$$
This problem first arose
in Stechkin’s investigations in 1965. The statement of the problem, first important results, and
solutions to this problem for differential operators of small orders are presented in~\cite{Stechkin1967}. For a survey of further results on this problem see~\cite{ Arestov1996}. 

Inequalities that estimate norms of the intermediate derivatives of univariate or multivariate functions using the norms of the functions and their derivatives of higher order play an important role in many branches of Analysis and its applications. It appears that the richest applications are obtained from sharp inequalities of this kind, which attracts much interest to the inequalities with the smallest possible constants. For univariate functions, the results by Landau~\cite{Landau}, and Kolmogorov~\cite{Kolmogorov1939}, are among the brightest ones in this topic. Inequalities of this kind are often called the Landau-Kolmogorov type inequalities. A survey on the results for univariate and multivariate functions for the case of derivatives of integer order and further references can be found in~\cite{
Arestov1996, BKKP, Konovalov1978,
BuslaevTikhomirov1979,Timoshin1995,
Timofeev1985, BKP1997, Babenko2000}. 

In many questions of Analysis, the necessity to deal with fractional derivatives~\cite{Samko} , hypersingular and more general integral operators occurs. Some results and further references  related to Landau-Kolmogorov type inequalities for such operators can be found in \cite{Geisberg1965,Arestov1979,MagTikh1981,BCh2001, BabenkoPichugov2007,Babenko07,BabenkoPichugov2010,babenko2010,BPP2010,BP2012-nabla,BP2012-delta,BPP2014-delta,BP2016-delta, MBDK}.

The Stechkin problem, in turn, is intimately connected to Landau-Kolmogorov type inequalities. We briefly describe this connection, since such scheme of arguments, (which is minor generalization of  Stechkin's techniques) will be used during the proof of the main results of this article. 

Knowing the value of the quantity in~\eqref{uklonenie}, one can immediately write a Landau-Kolmogorov  type inequality in the additive form:
$$
    \forall x\in Q, \quad \|Ax\|\le \|Ax-Sx\|+\|Sx\|\le U(A,S;Q)+\|S\|\cdot \|x\|.
$$
If in addition there exist an operator $\overline{S}\in \mathcal{L}(X,Y)$ and a  bounded sequence  of elements $\{x_n\}\subset Q$ such that
\begin{equation}\label{i::equality}
    \|Ax_n\|=\left(U(A,\overline{S};Q)+\|\overline{S}\|\cdot \|x_n\|\right)(1+o(1)) \text{ as } n\to\infty,
\end{equation}
then 
$$
E_{\left\|\overline{S}\right\|}(A, Q)=U(A,\overline{S};Q).
$$
Indeed, for each operator $S$ with $\|S\|\le \|\overline{S}\|$, one has 
\begin{equation}\label{i::upper-bound-5}
    \|Ax_n\|\le U(A,S;Q)+\|S\|\cdot \|x_n\|
    \le U(A,S;Q)+\|\overline{S}\|\cdot \|x_n\|.
\end{equation}
Passing to a limit as $n\to\infty$, due to boudedness of $\{x_n\}$,  equalities~\eqref{i::equality} and~\eqref{i::upper-bound-5}, we obtain
$
U(A,\overline{S};Q)\le U(A,S;Q),
$
as required. 

The article is organized as follows. In Section~\ref{sec::notations} we give necessary definitions and formulate our main results. 
In Section~\ref{sec::ostrowski} we compute norms of deviation of a hypersingular integral operator from  bounded operators of a particular family; this result can also be considered as  a sharp Ostrowski 
type inequality~\cite{Ostrowski1938} (see e.g. the survey~\cite{Dragomir17} for other results in this area and further references) for the integral with rather general weight, and is a generalization of the one from~\cite{Babenko20}.
In Section~\ref{sec::kolmogorov},  we obtain a sharp Landau-Kolmogorov type inequality for a large class of hypersingular integral operators, which in particular includes  the Riesz fractional derivatives. This result generalizes~\cite{BP2012-nabla}. Following the described above scheme, we also give a proof of our main result Theorem~\ref{th::StechkinProblem} in this section.
In Section~\ref{sec::applications} we give several applications of the obtained results.  In particular, we  solve the problem of computation of the modulus of continuity for the hypersingular integral operators, as well as the problem of recovery of such operators on the class of known with error elements.  Finally, in order not to shadow the main ideas of the proofs of our results, we transfer several auxiliary technical proofs into Appendices.

\section{Notations and the main results}\label{sec::notations}
For $x,y\in\Rd$ denote by $(x,y)$ the dot product of $x$ and $y$. Let $K\subset \Rd$ be an open convex symmetric with respect to $\theta$ bounded set with $\theta\in{\rm int\,} K$. Denote by $\mathbb{K}$ the family of all such sets $K$. For $K\in\mathbb{K}$ and $x\in \Rd$ denote by 
$|x|_K$ the norm of $x$ generated by the set $K$, i.e.
$$|x|_K:=\inf\{\lambda>0\colon x\in \lambda K\}.$$
If $K$ is the unit ball in $\Rd_p$, $p\in [1,\infty]$, then we write $|\cdot|_p$ instead of $|x|_K$.
For a set $K\in \setsClass$, denote by $\polarK$ its polar set, 
$$K^\circ = \left\{y\in \Rd\colon \sup\limits_{x\in K} (x,y)\leq 1\right\}.$$
It is well known that the norm $|\cdot|_{\polarK}$ is the dual to $|\cdot|_K$ norm, i.e. 
$$
|z|_{\polarK} = \sup \{(x,z)\colon |x|_K\leq 1\}.
$$ 
In particular for all $x,y\in\Rd$,
$
(x,y)\leq |x|_K\cdot |y|_{\polarK}.
$

Let $Q\subset \Rd$ be an open set. By  $W^{1,p}(Q)$, $p\in [1,\infty]$, we denote the Sobolev space of functions $f\colon Q\to \RR$ such that $f$ and all their (distributional) partial derivatives of the first order belong to $L_p(Q)$. Let $K\in\setsClass$ be fixed.  
Since all norms in the finite dimensional space $\Rd$ are equivalent, and obviously $|\nabla f|_1\in L_p(Q)$ for all $f\in W^{1,p}(Q)$, one has $|\nabla f|_K\in L_p(Q)$.
For $p\in[1,\infty]$ set 
$$
W^{K}_{p}(Q):=\left\{f\in W^{1,p}(Q)\colon \||\nabla f|_{\polarK}\|_{L_p(Q)}\leq 1\right\}.
$$

By $\conesClass$ we denote the set of all open convex cones generated by a finite number of points, i.e. the family of non-empty sets of the form
$$
 {\rm int\,}\left\{\sum_{k=1}^n c_k x_k\colon c_k\geq 0, k=1,\ldots, n\right\}\subset \Rd
$$
with $n\in\mathbb{N}$ and $x_1,\ldots, x_n\in\Rd$. Note that the family $\conesClass$ contains the whole space $\Rd$ as an element, and for each $C\in \conesClass$ and polytope $K\in\setsClass$ (by a polytope, we mean a convex hull of a finite set of points from $\Rd$) the set $
K\cap C$, is also a polytope.

Everywhere below we consider the case $d\geq 2$, $p\in (d,\infty]$ and $p'$ is such that $p^{-1}+(p')^{-1}=1$. We also fix a cone $C\in\conesClass$.

As the set $Q$, we consider either the sets of the form $\KC$ or finite unions of shifts of a set from $\setsClass$. Both types of sets satisfy the cone condition, and hence the imbedding of the class $W^{1,p}(Q)$ into the space of bounded continuous on $Q$ functions holds, see~\cite[Chapter~4]{adams2003}. Thus  the values $f(x)$, $x\in Q$ and $f\in W^{1,p}(Q)$ are well defined.

For $C\in \conesClass$ by $L_{\infty,p}^1(C)$, $p\in [1,\infty]$, we denote the space of functions $f\colon C\to \RR$ such that $f\in L_\infty(C)$ and all their (distributional) partial derivatives of the first order belong to $L_p(C)$. 
For $p\in[1,\infty]$ and $C\in \conesClass$ set 
$$
W_{\infty,p}^K=W^{K}_{\infty,p}(C):=\left\{f\in L_{\infty,p}^1(C)\colon \||\nabla f|_{\polarK}\|_{L_p(C)}\leq 1\right\}.
$$

In the article we approximate the following hypersingular integral operator $D_{K,w}\colon L_{\infty,p}^1(C)\to L_\infty(C)$
\begin{equation}\label{unboundedOperator}
    (D_{K,w} f)(x)
:=
\int_{C}w (|t|_K)(f(x)-f(x+t)) dt, x\in C,
\end{equation}
by bounded operators,  where the ''radial'' weight $w(|t|_K)$ is generated by a uni-variate function $w\colon (0,\infty)\to[0,\infty)$.
Note that in the case when $K$ is the unit ball in $\Rd$, $C=\Rd$ and $w (t)=t^{-(d+\alpha)}$, $0<\alpha < 1$, we obtain the Riesz derivative $D^\alpha$ of order $\alpha$,  see \cite[\S 26.7]{Samko}.  It is worth mentioning that the Riesz derivative $D^\alpha$ corresponds \cite[\S 25.4]{Samko} to the fractional degree $(-\Delta)^{\alpha/2}$ of the Laplace operator.
  
 In order to specify the class of considered weights, we need the following definitions.
 \begin{definition}
For $ 0\leq a < b\leq +\infty$ and $p\in[1,\infty]$, denote by $\weightedLp_p(a,b)$ the space of all  measurable functions $w\colon (a,b)\to\RR_+$ with finite norm
$$
\|w\|_{\weightedLp_p(a,b)} =
\begin{cases}
\left(\int_a^b t^{d-1} w^p(t)dt\right)^{1/p}, & p < \infty,\\
\esssup_{t\in (a,b)} t^{d-1}w(t),& p = \infty.
\end{cases}
$$
\end{definition}
Note that the condition $w\in \weightedLp_p(0,1)$ in the case $p\in [1,\infty)$ guarantees that the integral 
$
\int_{\KC} w^p(|t|_K)dt
$
exists for arbitrary polytope  $K\in\setsClass$ and cone $C\in \conesClass$, see Lemma~\ref{l::radialFunctionIntegral} below. 

\begin{definition}
For $h>0$ and $p\in[1,\infty]$ denote by $\weightsClass_p(0,h)$ the space of all non-negative functions $w\colon (0,h]\to\RR$ such that $w\in\weightedLp_ 1(u,h)$  for all $u\in (0,h)$ and the  function
\begin{equation}\label{defG}
g_{w,h}(u) = g_w(u) = \frac {1} {u^{d-1}} \int\limits_{u}^hw(t)t^{d-1}dt 
\end{equation}
belongs to $\weightedLp_p(0,h)$.
\end{definition}

\begin{definition}
For $h>0$ and $p\in[1,\infty]$ denote by $\weightsClass_p^*(0,h)$ the space of all non-negative functions $w\colon (0,h]\to \RR$ such that $w\in\weightedLp_ p(u,h)$  for all $u\in (0,h)$, for all $\nu$ close enough to $1$ from the left, the function
\begin{equation}\label{wNu}
    w_\nu (u) = \sup_{t\in [\nu u, u]}|w(t)-w(u)|
\end{equation}
belongs to $\weightsClass_p(0,h)$, and 
\begin{equation}\label{defDiffIsSmall1}
    \lim_{\nu\to 1-0}\|g_{w_\nu, h}\|_{\weightedLp_p(0,h)} = 0.
\end{equation}
\end{definition}
As it can be proved (using the arguments from the proof of Lemma~\ref{l::norm-of-truncated} below), the operator $D_{K,w}$ defined in~\eqref{unboundedOperator} is unbounded whenever the integral $\int_{0}^\infty w(\rho)\rho^{d-1}\, d\rho$ diverges and is bounded with norm equal to $2d\cdot \mes(\KC)\int_{0}^\infty w(\rho)\rho^{d-1}\, d\rho$ otherwise. In the statements of the theorems below we use the function $g_{w,h}$ that is defined in~\eqref{defG}.
The main result of the article is given by the following theorem. 

\begin{theorem}\label{th::StechkinProblem}
Let $p\in (d,\infty]$, $C\in\conesClass$, $K\in \setsClass$ be a polytope and  $w\in \weightsClass_{p'}(0,1)\cap \weightedLp_1(1,\infty)$, or $K\in \setsClass$ and  $w\in \weightsClass_{p'}^*(0,1)\cap \weightedLp_1(1,\infty)$. For all 
$$
N\in \begin{cases}
(0,\infty), & D_{K,w} \text{ is unbounded,} \\
\left(0,  \|D_{K,w}\|\right), & \text {otherwise,}
\end{cases}
$$
let $h_N$ be such that
$
    2d\cdot \mes(\KC)\int_{h_N}^\infty w(\rho)\rho^{d-1}\, d\rho =N.
$ Then
$$
E_N(D_{K,w}, W_{\infty,p}^K)=(d\cdot \mes(\KC))^{\frac{1}{p'}}
\|g_{w,h_N}\|_{\weightedLp_{p'}(0,h_N)}.
$$
Moreover, the extremal operator is
\begin{equation}\label{truncatedOperator}
( D_{K, w, h_N}f)(x):=\int_{C\setminus h_NK}w (|t|_K)(f(x)-f(x+t)) dt, x\in C.
\end{equation}
\end{theorem}
In order to use the outlined above scheme of the proof for Theorem~\ref{th::StechkinProblem}, we prove several results, which are interesting on their own.
First, we compute the norms of the deviation $D_{K,w}- D_{K,w,h}$ between the operators~\eqref{unboundedOperator} and~\eqref{truncatedOperator} (with arbitrary $h>0$ instead of $h_N$). This result can also be considered as a sharp Ostrowski type inequality. 
\begin{theorem}\label{th::mainInequality}
Let $p\in (d,\infty]$, $C\in\conesClass$, $K\in \setsClass$ be a polytope, $h>0$ and  $f\in W^{1,p}(h \KC)$. For each $w\in \weightsClass_{p'}(0,h)$ one has
\begin{multline}\label{mainInequality}
\left|\int\limits_{\,h\KC}w(|y|_K)[f(y) - f(\theta)]
dy\right|
\\ \leq 
\left(d\cdot\mes(\KC)\right)^{\frac 1{p'}} \left\|g_{w,h} \right\|_{\weightedLp_{p'}(0,h)} 
\||\nabla f|_{\polarK}\|_{L_p(h\KC)}.
\end{multline}
The inequality is sharp. It becomes equality for the functions $a\cdot f_e + b$, where $a,b\in\RR$ and
\begin{equation}\label{extremalFunction}
	f_e(y)=f_{e,h}(y)= \int\limits_0^{|y|_{K}}g_{w,h}^{p'-1}(u)du, \, y\in h\KC.
\end{equation}
If additionally $w\in \weightsClass_{p'}^*(0,h)$, then inequality~\eqref{mainInequality} holds and is sharp for arbitrary $K\in\setsClass$.
\end{theorem}

Next we prove the following sharp Landau-Kolmogorov type inequality in the additive form.
\begin{theorem}\label{th::additiveKolmogorov}
Let $p\in (d,\infty]$, $C\in \conesClass$, $K\in \setsClass$ be a polytope, $h>0$, $w\in \weightsClass_{p'}(0,h)\cap \weightedLp_1(h,\infty)$. For each  $f\in L_{\infty,p}^1(C)$ one has
\begin{multline}\label{Kolm-add-th}
\|D_{K,w} f\|_{L_\infty(C)}
 \le 
 \left(d\cdot\mes(\KC)\right)^{\frac 1{p'}} \left\|g_{w,h} \right\|_{\weightedLp_{p'}(0,h)}\cdot \||\nabla f|_{\polarK}\|_{L_p(C)}
\\+2d \cdot \mes(\KC) \left(\int\limits_h^\infty w(\rho)\rho^{d-1}\, d\rho\right)\|f\|_{L_\infty(C)}.
\end{multline}
The inequality is sharp.  It becomes equality on the functions $a\cdot \psi$, where $a\in\RR$ and 
\begin{equation}\label{extremalPsi}
  \psi (t) = \psi_{K,w, h}(t)
=
\begin{cases}
\frac{1}{2}\int\limits_0^{h}g_{w,h}^{p'-1}(u)du -\int\limits_0^{|t|_K}g_{w,h}^{p'-1}(u)du
, &|t|_K\le h,\\
-\frac{1}{2}\int\limits_0^{h}g_{w,h}^{p'-1}(u)du, &|t|_K> h.
\end{cases}  
\end{equation}
If additionally $w\in \weightsClass_{p'}^*(0,h)$, then inequality \eqref{Kolm-add-th} holds and is sharp for arbitrary $K\in \setsClass$.
\end{theorem}

If in addition the weight $w$ is a power function, then Theorem~\ref{th::additiveKolmogorov} implies the following Landau-Kolmogorov type inequality in the multiplicative form.

\begin{corollary}\label{c::multiplicativeKolmogorov}
Let $p\in (d,\infty]$, $C\in \conesClass$, $K\in \setsClass$, and 
\begin{equation}\label{powerW}
    w(t) = \frac {1}{t^{d+\gamma}}, t>0, \text{ with }  0<\gamma < 1-\frac dp.
\end{equation}
 Then  for each  $f\in L_{\infty,p}^1(C)$ one has
\begin{equation}\label{multInequality}
\|D_{K,w} f\|_{L_\infty(C)}\le \frac{X^{p'} + YZ}{X^{(p'-1)\alpha}\cdot Z^{1-\alpha}}
\cdot\|f\|_{L_\infty(C)}^{1-\alpha}\cdot
\||\nabla f|_{\polarK}\|_{L_p(C)}^{\alpha},
\end{equation}
where
$
\alpha = \frac{p\gamma}{p-d}
$, $X = \left(d\cdot\mes(\KC)\right)^{\frac 1{p'}} \left\|g_{w,1}\right\|_{\weightedLp_{p'}(0,1)}$,
$$
 Y=2d \cdot \mes(\KC) \int\limits_1^\infty w(\rho)\rho^{d-1}d\rho = \frac{2d \cdot \mes(\KC)}{\gamma}
 $$
 and $Z =  \frac{1}{2}\int\limits_0^{1}g_{w,1}^{p'-1}(u)du.
$
The inequality is sharp. If $K$ is a polytope, then for each $h>0$ it becomes equality on the function $\psi_{K,w,h}$ defined in~\eqref{extremalPsi}.
\end{corollary}

\section{Ostrowski type inequality}\label{sec::ostrowski}
\subsection{Auxiliarly results}
 We need the following lemma, which follows from the results in~\cite[Chapter 6.9]{lieb2001}.
\begin{lemma}\label{l::LagrangeTh}
Suppose $Q\subset \Rd$ is an open convex set, $f\in W^{1,p}(Q)$ and $x,y\in Q$. Then
$$f(y)-f(x)=\int\limits_0^1\left(y-x,\nabla f[(1-t)x+ty]\right)dt.$$
\end{lemma}

The following consequence of the coarea (see e.g.~\cite[Theorem~3.2.12]{federer1969}) formula is also needed.
\begin{lemma}\label{l::curvelinearCoords}
Let $C\in \conesClass$,  $K\in\setsClass$ be a polytope, $\gamma$ be a face of $\KC$ that does not contain the origin $\theta$ and  $\delta$ be the distance between $\theta$
and the plane that contains $\gamma$. Set
\begin{equation}\label{GammaSet}
\Gamma:= {\rm conv}(\{\theta\}\cup \gamma).
\end{equation}
 For all $h>0$ and integrable $g\colon h\KC\to\mathbb{R}$ 
$$\int\limits_{h\Gamma}g(x)dx=\delta\int\limits_0^h \rho^{d-1}	\int\limits_{\gamma}g(\rho y)dy d\rho.$$
\end{lemma}
\begin{proof}
Let $n$ be the unit external normal of the face $\gamma$. Consider the function $u(x) = \frac{1}{\delta}(n,x)$. Then $|\nabla u(x)|_2 = \frac 1\delta$ for all $x\in\Rd$, $\gamma\subset\{x\in\Rd\colon u(x) =1\}$, and using the coarea formula we obtain 
\begin{gather*}
\int\limits_{h\Gamma}g(x)dx 
= \delta \int\limits_{h\Gamma}g(x)|\nabla u(x)|_2dx 
=
\delta \int\limits_{0}^h\int\limits_{\rho\cdot\gamma} g(y)dyd\rho 
=
{\delta}\int\limits_0^h \rho^{d-1}	\int\limits_{\gamma}g(\rho z)dz d\rho.
\end{gather*}
\end{proof}

We need the following lemma.
\begin{lemma}\label{l::radialFunctionIntegral}
If $C\in \conesClass$, $h>0$, $p\in [1,\infty)$ and $w\in \weightedLp_p(0,h)$, then for each polytope $K\in \setsClass$ 
$$
\int\limits_{h\KC} w^p(|y|_K)dy = d\cdot \mes(\KC)\int\limits_{0}^hw^p(t) t^{d-1}dt.
$$
\end{lemma}
\begin{proof}
Let $\gamma_1,\ldots, \gamma_k$ be all faces of the polytope $\KC$ that do not contain $\theta$, $\delta_1,\ldots, \delta_k$ be the distances from corresponding planes to the origin. Set $\Gamma_s = {\rm conv}(\{\theta\}\cup\gamma_s)$, $s=1,\ldots, k$. Then using Lemma~\ref{l::curvelinearCoords}, we obtain
\begin{multline*}
\int\limits_{h\KC}w^p(|x|_{K})dx  
=
\sum\limits_{s=1}^k\int\limits_{h\Gamma_s}w^p(|x|_{K})dx 
=
\sum\limits_{s=1}^k\delta_s\int\limits_0^h\rho^{d-1}\int\limits_{\gamma_s}w^p(|\rho z|_K)dzd\rho
\\=
\sum\limits_{s=1}^k d\cdot\mes\Gamma_s\int\limits_0^h\rho^{d-1} w^p(\rho )d\rho
=
d\cdot \mes(\KC)\int\limits_{0}^hw^p(\rho) \rho^{d-1}d\rho.
\end{multline*}
\end{proof}

\begin{lemma}\label{l::faceNormalNorm}
Let $K\in \setsClass$ be a polytope, $\gamma$ be its face, $\delta$ be the distance between the origin $\theta$ and the plane that contains $\gamma$, and $n$ be the external unit normal of the face $\gamma$. Then $|n|_{\polarK} = \delta$.
\end{lemma}
\begin{proof}
Since $K$ is convex, it lies from one side of the plane that contains $\gamma$, hence
$$
	|n|_{\polarK} =  \sup \{(x,n)\colon |x|_K\leq 1\}\leq \delta.
$$
On the other hand, if $y\in \gamma$, then $|y|_K = 1$ and  $(y,n) = \delta$.
\end{proof}

\begin{lemma}\label{l::extremalFunction}
Let $K\in \setsClass$ be a polytope, $h>0$, $p\in (d,\infty]$,   $g\in \weightedLp_p(0,h)$ and
\begin{equation*}
	f(y)= \int\limits_0^{|y|_{K}}g(u)du,\, y\in hK.
\end{equation*}
If $\gamma$ is a face of $K$ with external unit normal $n$,  $\delta$ is the distance between the origin $\theta$ and the plane that contains $\gamma$, and $\Gamma$ is defined by~\eqref{GammaSet}, then for  $y\in {\rm int\,}h\Gamma$
\begin{equation}\label{nablaFe}
\nabla f (y) = \frac{g(|y|_K)}{\delta}\cdot n.
\end{equation}
Moreover, for almost all $y\in hK$
\begin{equation}\label{nablaFeNorm}
|\nabla f(y)|_{\polarK} = g(|y|_K).
\end{equation}
If $p<\infty$, then $f \in W^{1,p}(h K)$.
\end{lemma}
\begin{proof}
Note that for all $y\in {\rm int\,}h\Gamma$,
\begin{equation}\label{yNorm}
|y|_K = \frac 1{\delta} (y,n).
\end{equation}
Let  $n=(n_1,\ldots, n_d)$ and for $j\in \{1,\ldots, d\}$,  $e_j$ be the $j$-th element of the basis in $\Rd$. Using~\eqref{yNorm}, we obtain
\begin{gather*}
\frac{\partial f}{\partial x_j}(y) 
=
\lim\limits_{\eta\to 0} \frac{1}{\eta}\int\limits_{|y|_K}^{|y+\eta e_j|_K}g(u)du
=
\lim\limits_{\eta\to 0} \frac{1}{\eta}\int\limits_{|y|_K}^{|y|_K + \frac{\eta}{\delta} n_j}g(u)du
=g(|y|_K)\cdot \frac{n_j}{\delta}.
\end{gather*}
This implies~\eqref{nablaFe}, and hence, by Lemma~\ref{l::faceNormalNorm}, equality~\eqref{nablaFeNorm}. Lemma~\ref{l::radialFunctionIntegral} implies that $|\nabla f(\cdot)|_{\polarK}\in L_p(hK)$, if $p<\infty$.
\end{proof}

\begin{lemma}\label{l::changeOrderTwice}
Let $C\in \conesClass$, $p\in (d,\infty]$, $K\in \setsClass$ be a polytope, $h>0$ and  $f\in W^{1,p}(h \KC)$. For each $w\in \weightsClass_{p'}(0,h)$ one has
\begin{equation}\label{changeOrderEquality}
\int\limits_{\,h\KC}\int\limits_0^1w(|y|_{K})|y|_{K}|\nabla f(ty)|_{\polarK}dtdy
=
\int\limits_{h\KC}|\nabla f(y)|_{\polarK} g_{w,h}(|y|_K)  dy.
\end{equation}
\begin{proof}
Let $\gamma$, $n$, $\delta$ and $\Gamma$ be as in Lemma~\ref{l::curvelinearCoords}. We split the set $h\KC$ into the simplices $h\Gamma$ and evaluate each of the obtained from the left-hand side of~\eqref{changeOrderEquality} summands. 
\begin{gather*}
\notag \int\limits_{\,\,h\Gamma}\int\limits_0^1w(|y|_K)|y|_{K}|\nabla f(ty)|_{\polarK}dtdy
\\ \notag 
\text{(by Lemma~\ref{l::curvelinearCoords})} 
=
\delta\int\limits_{\gamma}\int\limits_{0}^h\int\limits_0^1w(|\rho z|_K)\rho^{d-1}|\rho z|_K|\nabla f(\rho tz)|_{\polarK} dtd\rho dz
\\ \notag  \text{(substitution $t = \frac{s}{\rho}$; note that $|\rho z|_K= \rho$ for all $z\in \gamma$)}
\\ \notag =
\delta\int\limits_{\gamma}\int\limits_{0}^h\int\limits_0^\rho w(\rho)\rho^{d-1}|\nabla f(s z)|_{\polarK}dsd\rho dz
\\  \notag \text{(change the order of integration in two internal integrals)}
\\  =
\delta\int\limits_{\gamma}\int\limits_{0}^h|\nabla f(s z)|_{\polarK}\int\limits_s^h w(\rho)\rho^{d-1}d\rho dsdz 
=
\delta\int\limits_{0}^hs^{d-1}\int\limits_{\gamma} |\nabla f(s z)|_{\polarK}g_{w,h}(|sz|_K) dzds
\\ \notag\text{(by Lemma~\ref{l::curvelinearCoords})}
 =
\int\limits_{h\Gamma}|\nabla f(y)|_{\polarK} g_{w,h}(|y|_K) dy.
\end{gather*}
Summing all such equalities we obtain the required.
\end{proof}
\end{lemma}

We need the following technical lemma. Its proof will be given in Appendix~\ref{app::secondSummandProof}.
\begin{lemma}\label{l::secondSummand}
Let $C\in \conesClass$, $p\in (d,\infty]$, $K\in \setsClass$ $h>0$, $f\in W^{1,p}(h \KC)$ and  $w\in \weightsClass_{p'}^*(0,h) $. For each $\varepsilon> 0$ choose a polytope $K_\varepsilon\in\setsClass$, $K\supset K_\varepsilon$ such that 
\begin{equation}\label{K_epsilon}
\max\left\{\sup\limits_{x\in K}\inf\limits_{y\in K_\varepsilon}|x-y|_2, \sup\limits_{x\in K_\varepsilon}\inf\limits_{y\in K}|x-y|_2\right\}<\varepsilon.    
\end{equation}
Then as $\varepsilon\to 0$,
\begin{equation}\label{secondSummand}
\int\limits_{hK_\varepsilon\cap C}[w(|y|_{K_\varepsilon}) - w(|y|_K)][f(y) - f(\theta)]dy\to 0 
\end{equation}
and
\begin{equation}\label{thirdSummand}
\int\limits_{\,h(K\setminus K_\varepsilon)\cap C}w(|y|_{K})[f(y) - f(\theta)]dy \to 0.
\end{equation}
\end{lemma}

\subsection{Proof of Theorem~\ref{th::mainInequality}}
\begin{proof}
We prove the statement of the theorem in the case, when $K$ is a polytope first. Using Lemma~\ref{l::LagrangeTh}, Lemma~\ref{l::changeOrderTwice} and the H\"{o}lder inequality, we obtain
\begin{multline}\label{firstIneq}
\left|\int\limits_{\,h\KC}w(|y|_K)[f(y) - f(\theta)]
dy\right|
=
\left|\int\limits_{\,h\KC}w(|y|_K)\int\limits_0^1(y,\nabla f(ty))dtdy\right|
\\  \leq 
\int\limits_{h\KC}\int\limits_0^1w(|y|_K)|y|_{K}|\nabla f(ty)|_{\polarK}dtdy = \int\limits_{h\KC}|\nabla f(y)|_{\polarK} g_{w,h}(|y|_K)  dy
\\\leq 
\|\,|\nabla f|_{\polarK}\,\|_{L_p(h\KC)} \left\|g_{w,h}(|\cdot|_K) \right\|_{L_{p'}(h\KC)}
\end{multline}
and inequality~\eqref{mainInequality} follows from Lemma~\ref{l::radialFunctionIntegral}.

The fact that the function $f_e$ defined by equality~\eqref{extremalFunction} belongs to the class $W^{1,p}(h\KC)$, follows from equality $(p'-1)p =p'$ and Lemmas~\ref{l::radialFunctionIntegral} and~\ref{l::extremalFunction}. 

In order to prove that the function $f_e$ turns inequality~\eqref{mainInequality} into equality, we examine the inequalities that were used during the proof. First of all note that inside each set ${\rm int\,}\Gamma$ that partition the set $h\KC$ (see Lemma~\ref{l::curvelinearCoords}), due to Lemma~\ref{l::extremalFunction}, we obtain
$$
(y,\nabla f_e(ty)) 
=
(y,n) \frac{g_{w,h}^{p'-1}(|ty|_K)}{\delta} = |y|_K\cdot |\nabla f_e(ty)|_{\polarK},
$$
and the first inequality in~\eqref{firstIneq} becomes equality for $f_e$. Next, since $p(p'-1) = p'$, using Lemma~\ref{l::extremalFunction} we obtain
$
|\nabla f_e(y)|_{\polarK}^p = g_{w,h}^{p'}(|y|_K),
$
and hence the second inequality in~\eqref{firstIneq} also becomes equality for $f_e$. Hence $f_e$ turns inequality~\eqref{mainInequality} into equality. The theorem is proved in the case, when $K$ is a polytope. 

Let now $w\in \weightsClass_{p'}^*(0,h)$ and $K$ be an arbitrary set from $\setsClass$. For $\varepsilon> 0$ choose a polytope $K_\varepsilon\in\setsClass$, $K\supset K_\varepsilon$ such that inequality~\eqref{K_epsilon} holds. For arbitrary $f\in W^{1,p}(h\KC)$
\begin{multline}\label{arbitraryK}
\left|\int\limits_{\,h\KC}w(|y|_K)[f(y) - f(\theta)] dy\right| 
=
\left|\int\limits_{\,hK_\varepsilon\cap C}w(|y|_{K_\varepsilon})[f(y) - f(\theta)]
dy\right.
\\ 
+
\left.\int\limits_{hK_\varepsilon\cap C}[w(|y|_K) - w(|y|_{K_\varepsilon})][f(y) - f(\theta)]dy
+
\int\limits_{h(K\setminus K_\varepsilon)\cap C}w(|y|_{K})[f(y) - f(\theta)]dy\right|.
\end{multline}
Since $K_\varepsilon\subset K$, one has $\polarK_\varepsilon\supset \polarK$, and hence
\begin{equation}\label{polarNormsComparison}
    |y|_{\polarK_\varepsilon}\leq |y|_{\polarK},\, y\in \Rd.
\end{equation}
Using inequalities~\eqref{polarNormsComparison} and~\eqref{mainInequality}, we obtain as $\varepsilon\to 0$
\begin{multline}\label{firstSummand}
\left|\int\limits_{\,hK_\varepsilon\cap C}w(|y|_{K_\varepsilon})[f(y) - f(\theta)]
dy\right|
\\\leq 
\left(d\cdot\mes (K_\varepsilon\cap C)\right)^{\frac 1{p'}} \left\|g_{w,h} \right\|_{\weightedLp_{p'}(0,h)} 
\||\nabla f|_{\polarK_\varepsilon}\|_{L_p(hK_\varepsilon\cap C)}
\\  \leq 
\left(d\cdot\mes (K_\varepsilon\cap C)\right)^{\frac 1{p'}} \left\|g_{w,h} \right\|_{\weightedLp_{p'}(0,h)} 
\||\nabla f|_{\polarK}\|_{L_p(hK_\varepsilon\cap C)}
\\ = 
(1+o(1))\left(d\cdot\mes(\KC)\right)^{\frac 1{p'}} \left\|g_{w,h} \right\|_{\weightedLp_{p'}(0,h)} 
\||\nabla f|_{\polarK}\|_{L_p(h\KC)}.
\end{multline}

The second and third summands on the right-hand side of~\eqref{arbitraryK} tend to zero as $\varepsilon\to 0$ due to Lemma~\ref{l::secondSummand}.
Thus we obtain that inequality~\eqref{mainInequality} holds for arbitrary $f\in W^{1,p}(h\KC)$.
In order to see that inequality~\eqref{mainInequality} is sharp, it is enough to notice that in~\eqref{firstSummand} for $f = f_e$ we actually have
\begin{multline*}
\left|\int\limits_{\,hK_\varepsilon\cap C}w(|y|_{K_\varepsilon})[f(y) - f(\theta)] dy\right|
\\=
(1+o(1))\left(d\cdot\mes(\KC)\right)^{\frac 1{p'}} \left\|g_{w,h} \right\|_{\weightedLp_{p'}(0,h)} 
\||\nabla f|_{\polarK}\|_{L_p(h\KC)},
\end{multline*}
and hence the constant on the right-hand side of~\eqref{mainInequality} can't be made smaller.
\end{proof}
\section{On Landau--Kolmogorov type inequalities}\label{sec::kolmogorov}
\subsection{Auxiliary results} 

We need the following lemmas. Their proofs are given in Appendices~\ref{app::intw(|t|)EstimateProof} and~\ref{app::fourthSummandProof}.
\begin{lemma}\label{l::intw(|t|)Estimate}
Let $0<h<H$, $w\in \weightedLp_1(h,H)$, sets $K_\varepsilon\in \setsClass$, $\varepsilon\in [0,1]$, be such that $K_{\varepsilon_1}\subset K_{\varepsilon_2}$ provided $\varepsilon_1<\varepsilon_2$, and $A\subset \Rd$ be a measurable set such that $|t|_{K_\varepsilon}\in [h,H]$ for all $t\in A$ and $\varepsilon\in [0,1]$. Then for each $\gamma>0$ there exists $\nu>0$ such that for all $\varepsilon\in [0,1]$ and $B\subset A$ with $\mes  B<\nu$ one has
$$
\int_B w(|t|_{K_\varepsilon})dt < \gamma.
$$
\end{lemma}
\begin{lemma}\label{l::fourthSummand}
Let $K$ be an arbitrary set from $\setsClass$. For each $\varepsilon \in (0,1]$ choose a polytope $K_{\varepsilon}\in \setsClass$, such that $K_{\varepsilon}\supset K$ and inequality~\eqref{K_epsilon} holds. If $K_{\varepsilon_1}\subset  K_{\varepsilon_2}$ provided $\varepsilon_1 < \varepsilon_2$,
 $h>0$ and $w\in \weightedLp_1(h,+\infty)$, then 
$$
\int\limits_{\Rd\setminus hK_\varepsilon}|w(|t|_K)-
w(|t|_{K_\varepsilon})|dt 
=
o(1)\text{ as } \varepsilon\to 0.
$$
\end{lemma}

The following statement holds.
\begin{lemma}\label{l::norm-of-truncated}
Let $C\in \conesClass$, $h>0$, $K\in \setsClass$ and $w\in \weightedLp_1(h,+\infty)$. Then the operator
$$
D_{K, w, h}\colon L_\infty(C)\to  L_\infty(C)
$$ defined by~\eqref{truncatedOperator} is bounded and 
$$
\|D_{K, w, h}\|=2d\cdot \mes(\KC)\int\limits_h^\infty \omega (\rho)\rho^{d-1}\, d\rho.
$$
If $f\in L_\infty(C)$ is such that $D_{K, w, h}f$ is  continuous at $\theta$ and such that $f(\theta) = \|f\|_{L_\infty(C)}$ and $f(x) = -\|f\|_{L_\infty(C)}$ for all $x\in C\setminus hK$, then 
\begin{equation}\label{trunkatedOperatorNorm}
    \|D_{K, w, h}f\|_{L_\infty(C)} = D_{K, w, h}f(\theta).
\end{equation}
\end{lemma}

\begin{proof}
We prove the statement of the lemma in the case when $K$ is a polytope first. For $f\in L_\infty (C)$  and $x\in C$
\begin{multline}\label{5.1}
 |D_{K, w, h}f(x)|=\left |\,\int\limits_{C\setminus hK}w(|t|_K)(f(x)-f(x+t))\, dt\right |
\\\le
2\|f\|_{L_\infty(C)}\int\limits_{C\setminus hK}w(|t|_K)\, dt.
\end{multline}
We split the set $h\KC$ into simplices $h\Gamma_i$, $i=1,\ldots, n$, and by $C_i$ denote the cone that corresponds to $h\Gamma_i$. As in the proof of Lemma~\ref{l::radialFunctionIntegral}, we obtain
\begin{gather*}
\int\limits_{C\setminus hK}w(|t|_K)\, dt
=
\sum\limits_{i=1}^{n}\int\limits_{C_i\setminus h\Gamma_i}w(|t|_K)\, dt
=
\sum\limits_{i=1}^{n}\delta_i\int\limits_h^\infty \rho^{d-1}\int\limits_{\gamma_i}w(|y\rho|_K)\, dy
\\
=
\sum\limits_{i=1}^{n} \delta_i\cdot\mes \Gamma_i\int\limits_h^\infty \rho^{d-1}w(\rho)\, d\rho
=
 d \cdot\mes(\KC)\int\limits_h^\infty \rho^{d-1}w(\rho)\, d\rho.
\end{gather*}
Using \eqref{5.1}, we obtain the inequality
\begin{gather}\label{5.2}
\|D_{K, w, h}f\|_{L_\infty(C)}\le 2\|f\|_{L_\infty(C)}d\cdot  \mes(\KC)\int\limits_h^\infty \rho^{d-1}w(\rho)\, d\rho.
\end{gather}
An example of a function $f$ that satisfies the described in the statement of the lemma conditions can easily be constructed (the functions defined by~\eqref{extremalPsi} are of this kind under additional assumption regarding $w$, see Lemma~\ref{l::derivativesContinuity} below). Equality~\eqref{trunkatedOperatorNorm} can be obtained by direct calculations.

 Let now $K$ be an arbitrary set from $\setsClass$. Set $K_0 = K$ and choose a family of polytopes $K_{\varepsilon}\in \setsClass$,   $\varepsilon \in (0,1]$, that satisfy the conditions of  Lemma~\ref{l::fourthSummand}. 
 For all $f\in L_\infty (C)$ and $x\in C$,
 \begin{multline}\label{5.2.1}
 |D_{K,w, h}f(x)|
 =
 \left |\,\int\limits_{C\setminus hK}w(|t|_K)(f(x)-f(x+t)) dt\right |
 \\=
 \left|\,\int\limits_{C\setminus hK_\varepsilon}w(|t|_{K_\varepsilon})(f(x)-f(x+t))\, dt
 +
 \int\limits_{h(K_\varepsilon\setminus K)\cap C}w(|t|_K)(f(x)-f(x+t))\, dt\right.
\\
+\left.\int\limits_{C\setminus hK_\varepsilon}(w(|t|_K)-
w(|t|_{K_\varepsilon}))(f(x)-f(x+t))\, dt
\right|.
\end{multline}
Next we estimate each of the summands on the right-hand side of equality~\eqref{5.2.1}.

By inequality \eqref{5.2} as $\varepsilon \to 0$
\begin{multline*}
\left |\,\int\limits_{C\setminus hK_\varepsilon}w(|t|_{K_\varepsilon})(f(x)-f(x+t))\, dt\right|
\le 
2\|f\|_{L_\infty(C)}d\cdot \mes (K_\varepsilon\cap C)\int\limits_h^\infty \rho^{d-1}w(\rho)\, d\rho\\
= (1+o(1))2\|f\|_{L_\infty(C)}d\cdot \mes(\KC)\int\limits_h^\infty \rho^{d-1}w(\rho)\, d\rho.
\end{multline*}
Moreover, the  inequality in the latter formula becomes equality, if $f=f_\varepsilon$ is an extremal for the polytope $K_\varepsilon$ function.
Using Lemma~\ref{l::intw(|t|)Estimate} we obtain as $\varepsilon\to 0$
$$
\left|\,\int\limits_{h(K_\varepsilon\setminus K)\cap C}w(|t|_K)(f(x)-f(x+t)) dt\right|
\leq 
2\|f\|_{L_\infty(C)}\int\limits_{h(K_\varepsilon\setminus K)}w(|t|_K) dt
=
o(1).
$$
Finally, using Lemma~\ref{l::fourthSummand}, we obtain 
\begin{multline*}
\left|\,\int\limits_{C\setminus hK_\varepsilon}(w(|t|_K)-
w(|t|_{K_\varepsilon}))(f(x)-f(x+t))dt\right|
\\ \leq
2\|f\|_{L_\infty(C)}\int\limits_{\Rd\setminus hK_\varepsilon}|w(|t|_K)-
w(|t|_{K_\varepsilon})|dt 
=
o(1) \text {  as } \varepsilon\to 0.
\end{multline*}

\end{proof}
\begin{remark}\label{r::embeddedPolytopes}
During the proof of the estimate from below for arbitrary $K\in \setsClass$, we approximated $K$ by a family of polytopes $K_\varepsilon$ that contain $K$. If $w\in \weightedLp_1(h-\nu,\infty)$ for some $\nu>0$, then the same arguments can be applied to a family of polytopes $K_\epsilon\subset K$ that approximate $K$.
\end{remark}

\begin{lemma}\label{l::derivativesContinuity}
Let $C\in\conesClass$, $p\in (d,\infty]$, $K\in \setsClass$ be a polytope,  $h>0$ and $w\in \weightsClass_{p'}(0,h)\cap \weightedLp_1(h,\infty)$. Then 
$D_{K,w}\psi$ and $D_{K,w,h}\psi$ are continuous at $\theta$, where $\psi$ is defined in~\eqref{extremalPsi}.
\end{lemma}
\begin{proof}
The function $\psi$ is continuous on $\Rd$, and since it is constant outside of a compact set, it is uniformly continuous on $\Rd$. If $\varepsilon>0$ and $\delta>0$ is chosen in such a way that $|\psi(x)-\psi(y)|<\varepsilon$ whenever $|x-y|_2<\delta$, then for all $|z|_2<\delta$
\begin{gather*}
    |D_{K,w, h}\psi (\theta) - D_{K,w, h} \psi(z)|
    =
    \left |\,\int\limits_{C\setminus hK}w(|t|_K)(\psi (\theta)-\psi(t) - \psi(z) + \psi (t+z)) dt\right|
    \\ \leq 
     \left |\,\int\limits_{C\setminus hK}w(|t|_K)(\psi (\theta)- \psi (z)) dt\right| 
    +  \left |\,\int\limits_{C\setminus hK}w(|t|_K)(\psi (t+z)-\psi (t) ) dt\right|
    \\\leq 
    2\varepsilon\int\limits_{C\setminus hK}w(|t|_K)dt,
\end{gather*}
which implies continuity of $D_{K,w, h}\psi$ at $\theta$.

In order to prove continuity of $D_{K,w}\psi$ at $\theta$, it is sufficient to prove continuity of $(D_{K,w}-D_{K,w,h})\psi$ at $\theta$. 
Applying Theorem~\ref{th::mainInequality} to the function $\psi(\cdot) - \psi(\cdot +z)$, $z\in h\KC$, we obtain
\begin{multline*}
  |(D_{K,w}-D_{K,w,h})\psi(\theta) - (D_{K,w}-D_{K,w,h})\psi(z)|
\\\leq
\|\,|\nabla \psi(\cdot) - \nabla\psi(\cdot + z) |_{\polarK}\,\|_{L_p(h\KC)} \left\|g_{w,h}(|\cdot|_K) \right\|_{L_{p'}(h\KC)}  \to 0\text{ as } z\to\theta.
\end{multline*}

\end{proof}
\subsection{Proof of Theorem~\ref{th::additiveKolmogorov} }
\begin{proof}
For arbitrary  $f\in L_{\infty,p}^1(C)$
$$
\|D_{K,w} f\|_{L_\infty(C)}
\le 
\|D_{K,w} f-D_{K,w, h}  f\|_{L_\infty(C)}+\|D_{K,w, h}  f\|_{L_\infty(C)}.
$$
For all  $y\in C$, using Theorem~\ref{th::mainInequality}, one has
 \begin{gather*}
|D_{K,w} f (y)-D_{K,w, h}  f(y)|
=
\left|\,\int\limits_{h\KC} w(|t|_K)(f(y)-f(y+t)) dt\right|
\\\le
  \left(d\cdot\mes(\KC)\right)^{\frac 1{p'}} \left\|g_{w,h} \right\|_{\weightedLp_{p'}(0,h)} 
 \||\nabla f|_{\polarK}\|_{L_p(y + h\KC)}
 \\ \leq 
 \left(d\cdot\mes(\KC)\right)^{\frac 1{p'}} \left\|g_{w,h} \right\|_{\weightedLp_{p'}(0,h)} 
\||\nabla f|_{\polarK}\|_{L_p(C)},
\end{gather*}
and hence
$$  \|D_{K,w} f-D_{K,w, h}  f\|_{L_\infty(C)} 
    \leq 
    \left(d\cdot\mes(\KC)\right)^{\frac 1{p'}} \left\|g_{w,h} \right\|_{\weightedLp_{p'}(0,h)} 
\||\nabla f|_{\polarK}\|_{L_p(C)}.
$$
Taking into account this inequality and Lemma~\ref{l::norm-of-truncated}, we obtain inequality~\eqref{Kolm-add-th}. 

Next we prove its sharpness. Let $K\in \setsClass$ be a polytope first. Consider the function defined in~\eqref{extremalPsi}.
Note that  
$$
\psi(\theta) 
=
\frac 12 \int\limits_0^{h}g^{p'-1}_{w,h}(u)du 
=
\|\psi\|_{L_\infty(C)} = \|\psi\|_{L_\infty(\Rd)},
$$
and $\psi(t) =- \|\psi\|_{L_\infty(C)}$ for all $t\in C\setminus hK$.  Hence by Lemmas~\ref{l::norm-of-truncated} and~\ref{l::derivativesContinuity} we obtain 
$$  
\|D_{K, w, h}\psi\|_{L_\infty(C)} = D_{K, w, h}\psi(\theta). 
$$
Moreover, the restriction of $\psi$ to $h\KC$ is extremal in Theorem~\ref{th::mainInequality}, and hence
$$ 
(D_{K,w} -D_{K,w, h})  \psi(\theta) = 
\left(d\cdot\mes(\KC)\right)^{\frac 1{p'}} \left\|g_{w,h} \right\|_{\weightedLp_{p'}(0,h)} 
\||\nabla \psi|_{\polarK}\|_{L_p(C)}.
$$
Taking into account Lemma~\ref{l::derivativesContinuity}, we obtain the required.

Let now $K\in \setsClass$ be an arbitrary set. From the conditions of the theorem it follows that $w\in \weightedLp_1(\nu,\infty)$ for all $\nu>0$. Consider a  monotone by inclusion family of polytopes $K_\varepsilon\subset K$ that approximate $K$. From Theorem~\ref{th::mainInequality}, it follows that the family of functions $\psi_{K_\varepsilon, w, h}$ defined in~\eqref{extremalPsi} (more precisely, the family of their restrictions to $h\KC$) is an extremal family for the operator $D_{K,w} - D_{K,w,h}$. Due to Remark~\ref{r::embeddedPolytopes}, the same family is extremal for the operator $D_{K,w,h}$. This implies shaprness of~\eqref{Kolm-add-th}.
\end{proof}
\subsection{Proof of Theorem~\ref{th::StechkinProblem}}
\begin{proof}
We use the scheme outlined in Section~\ref{sec::notations} with $A = D_{K,w}$ and $\overline{S} = D_{K,w,h_N}$.
In the case, when $K$ is a polytope, the extremal sequence of $\{x_n\}$
can be chosen to be constant, $x_n = \psi_{K,w,h_N}$, $n\in\mathbb{N}$. In the case, when $K\in\setsClass$ is arbitrary, one can consider the sequence of functions $\psi_{K_{\varepsilon_n}, w, h_N}$, where $K_{\varepsilon_n}\subset K$ approximate $K$, $n\in\mathbb{N}$.

\end{proof}
\subsection{Proof of Corollary~\ref{c::multiplicativeKolmogorov}}
Here we sketch the proof of Corollary~\ref{c::multiplicativeKolmogorov}. Technical details can be found in Appendix~\ref{app::homogenuity}.

First of all, Lemma~\ref{l::powerWIsGood} from Appendix~\ref{app::homogenuity} shows that Theorem~\ref{th::additiveKolmogorov} is applicable for the weight $w$ defined in~\eqref{powerW}.
Using homogenuity of $w$, the inequality from Theorem~\ref{th::additiveKolmogorov} can be rewritten as 
\begin{equation}\label{additiveInequalityForHomW}
  \|D_{K,w} f\|_{L_\infty(C)}
 \le 
 X \||\nabla f|_{\polarK}\|_{L_p(C)} h^{1+\beta + \frac{d}{p'}} + Y
\|f\|_{L_\infty(C)}h^{d+\beta}, 
\end{equation}
where $\beta:=-(d+\gamma)$. Setting 
\begin{equation}\label{multKolmogorovH}
   h = \left(\frac{\|f\|_{L_\infty(C)}}{ \||\nabla f|_{\polarK}\|_{L_p(C)}}\cdot \frac{X^{p'-1}}{Z}\right)^{\frac{p}{p-d}}
\end{equation}
in inequality~\eqref{additiveInequalityForHomW}, we obtain the required. In the case, when $K$ is a polytope, direct computations show that functions~\eqref{extremalPsi} turn the inequality into equality.
\section{Applications}\label{sec::applications}
The function
$$
  \Omega(\delta):=\sup_{ \substack{{f\in W_{\infty,p}^K}\\
\|f\|_{L_\infty(C)}\le \delta}}\|D_{w,K}f\|_{L_\infty(C)},\, \delta \ge 0,
$$
 is called the modulus of continuity of the operator $D_{w,K}$ on the set $ W_{\infty,p}^K$. The problem of finding the function $\Omega(\delta)$   for a given operator on a given  set can be considered as an abstract version of the problem on the Landau-Kolmogorov type inequality.

Another related problem is as follows. Let $\mathcal{O}$ be the set of all operators $$A\colon L_\infty(C)\to L_\infty(C)$$ and $\mathcal{R}\subset \mathcal{O}$.
 For $\delta\ge 0$ set
\begin{equation}\label{optimal-recovery}
    \mathcal{E}_\delta (\mathcal{R})=\inf_{T\in \mathcal{R}}
    \sup_{\substack{f\in  W_{\infty,p}^K,\, g\in L_\infty(C)\\ \|f-g\|_{L_\infty(C)}\le \delta}}\|D_{K,w}f-Tg\|_{L_\infty(C)}.
\end{equation}

The problem of optimal recovery of the operator $D_{K,w}$ with the help of methods from the class $\mathcal{R}$ on the set of elements $W_{\infty,p}^K$ given with error $\delta$, consists in finding the value of quantity~\eqref{optimal-recovery} and the optimal method of recovery $T$, if it exists.

Denote by $\mathcal{L}$ the set of all bounded linear operators $T\colon L_\infty(C)\to L_\infty(C)$. The following result shows the relation between the considered extremal properties.

\begin{corollary}\label{c::modulusOfContinuity}
Let $C\in\conesClass$, $p\in (d,\infty]$, $K\in \setsClass$ be a polytope, $h>0$, $w\in \weightsClass_{p'}(0,h)\cap \weightedLp_1(h,\infty)$. Then 
\begin{multline}\label{additiveModulusOfContinuity}
   \mathcal{E}_{\delta_h} (\mathcal{O}) 
   =
   \mathcal{E}_{\delta_h} (\mathcal{L}) 
   =
   \Omega(\delta_h) 
   \\ =  
\left(d\cdot\mes(\KC)\right)^{\frac 1{p'}} \left\|g_{w,h} \right\|_{\weightedLp_{p'}(0,h)}
+2d \cdot \mes(\KC) \left(\int\limits_h^\infty w(\rho)\rho^{d-1}\, d\rho\right)\delta_h,
\end{multline}
where $g_{w,h}$ is defined in~\eqref{defG} and
$$
\delta_h = \frac{\frac 12\int_0^hg_{w,h}^{p'-1}(u)du}{\left(d\cdot \mes(\KC)\|g_{w,h}\|_{\weightedLp_{p'(0,h)}}^{p'}\right)^{1/p}}.
$$
If $w\in \weightsClass_{p'}^*(0,h)$, then equality~\eqref{additiveModulusOfContinuity} holds for arbitrary $K\in \setsClass$.
If $w$ is the weight defined in~\eqref{powerW}, then for all $\delta>0$
  \begin{equation}\label{multModulusOfContinuity}
   \mathcal{E}_{\delta} (\mathcal{O}) 
   =
   \mathcal{E}_{\delta} (\mathcal{L}) 
   =
   \Omega(\delta) =  \frac{X^{p'} + YZ}{X^{(p'-1)\alpha}\cdot Z^{1-\alpha}}\cdot \delta^{1-\alpha},
  \end{equation}
  where the numbers $X,Y,Z$ and $\alpha$ are as in Corollary~\ref{c::multiplicativeKolmogorov}.
\end{corollary}
Having the tools developed above, this result can be proved by standard arguments, see e.g.~\cite[Theorem~7.1.2]{BKKP}. For completeness, we also give a direct proof in Appendix~\ref{app::modulusOfContinuityProof}.

\appendices
\section{Approximation of convex bodies by polytopes}
\subsection{Proof of Lemma~\ref{l::secondSummand}}\label{app::secondSummandProof}
\begin{proof}
We prove~\eqref{secondSummand} first. Since $\theta\in{\rm int\,}K_{\varepsilon}$  and $K_\varepsilon\subset K$, there exists $\nu= \nu(\varepsilon) \in (0,1)$ such that  for all $y\in K$,
$$ 
\nu |y|_{K_\varepsilon} \leq |y|_K \leq |y|_{K_\varepsilon},
$$
and $\nu(\varepsilon)\to 1$ as $\varepsilon\to 0$. 
 Applying Lemma~\ref{l::LagrangeTh},
 we obtain 
\begin{gather*}
\left|\int\limits_{\,hK_\varepsilon\cap C}[w(|y|_{K_\varepsilon}) - w(|y|_K)][f(y) - f(\theta)]dy\right| 
\\=
\left|\int\limits_{\,hK_\varepsilon\cap C}[w(|y|_{K_\varepsilon}) - w(|y|_K)]\int\limits_0^1(y,\nabla f(ty))dtdy\right|
\\  \leq 
\int\limits_{hK_\varepsilon\cap C}\int\limits_0^1|w(|y|_{K_\varepsilon}) - w(|y|_K)||y|_{K_\varepsilon}|\nabla f(ty)|_{\polarK_\varepsilon}dtdy 
\\\leq
\int\limits_{hK_\varepsilon\cap C}\int\limits_0^1w_\nu(|y|_{K_\varepsilon})|y|_{K_\varepsilon}|\nabla f(ty)|_{\polarK_\varepsilon}dtdy,
\end{gather*}
where $w_\nu$ is defined in~\eqref{wNu}. Applying Lemma~\ref{l::changeOrderTwice} and the H\"{o}lder inequality, we obtain the estimate
\begin{multline*}
\left|\int\limits_{\,hK_\varepsilon\cap C}[w(|y|_{K_\varepsilon}) - w(|y|_K)][f(y) - f(\theta)]dy\right| 
\\\leq
\| |\nabla f|_{\polarK_\varepsilon}\|_{L_p(hK_\varepsilon\cap C)} \left\|g_{w_\nu, h}(|\cdot|_{K_\varepsilon})\right\|_{L_{p'}(hK_\varepsilon\cap C)}.
\end{multline*}

Applying inequality~\eqref{polarNormsComparison}, Lemma~\ref{l::radialFunctionIntegral} and assumption~\eqref{defDiffIsSmall1}, we obtain the required.

Next we prove~\eqref{thirdSummand}.
\begin{gather*}
    \left|\int\limits_{\,h(K\setminus K_\varepsilon)\cap C}w(|y|_{K})[f(y) - f(\theta)]dy\right|
    \leq 
\int\limits_0^1\int\limits_{h(K\setminus K_\varepsilon)\cap C}w(|y|_K)|y|_{K}|\nabla f(ty)|_{\polarK}dydt
\\ \leq 
\int\limits_0^1
\left(\int\limits_{h(K\setminus K_\varepsilon)\cap C}(w(|y|_K)|y|_{K})^{p'}dy\right)^{\frac{1}{p'}}\left(\int\limits_{h(K\setminus K_\varepsilon)\cap C}|\nabla f(ty)|_{\polarK}^{p}dy\right)^{\frac{1}{p}}dt
\\=
\left(\int\limits_{h(K\setminus K_\varepsilon)\cap C}(w(|y|_K)|y|_{K})^{p'}dy\right)^{\frac{1}{p'}}
\int\limits_0^1 t^{-\frac dp}
\left(\int\limits_{ht(K\setminus K_\varepsilon)\cap C}|\nabla f(z)|_{\polarK}^{p}dz\right)^{\frac{1}{p}}dt.
\end{gather*}
Since $p>d$, and hence $-\frac dp > -1$, and $\mes ht(K\setminus K_\varepsilon)\leq \mes h(K\setminus K_\varepsilon) \to 0$, as $\varepsilon\to 0$, we obtain that 
$$
\int\limits_0^1 t^{-\frac dp}
\left(\int\limits_{ht(K\setminus K_\varepsilon)\cap C}|\nabla f(z)|_{\polarK}^{p}dz\right)^{\frac{1}{p}}dt \to 0,\text{ as } \varepsilon\to 0.
$$
The set $h(K\setminus K_\varepsilon)$ is separated from $\theta$ for all small enough $\varepsilon>0$, so the integral $\int\limits_{h(K\setminus K_\varepsilon)\cap C}(w(|y|_K)|y|_{K})^{p'}dy$ is finite, since $w\in \weightedLp_{p'}(\eta, h)$ for all $\eta>0$.
\end{proof}
\subsection{Proof of Lemma~\ref{l::intw(|t|)Estimate}}\label{app::intw(|t|)EstimateProof}
\begin{proof}
Since $K_\varepsilon\in \setsClass$ and $K_0\subset K_\varepsilon\subset K_1$ for all $\varepsilon\in [0,1]$, there exist $0<r<R$ such that $B_r\subset K_\varepsilon \subset B_R$, for all $\varepsilon\in [0.1]$, where  $B_r$ and $B_R$ are balls in $\Rd$ with center $\theta$ and radii $r$ and $R$ respectively. 

We show that all functions $w(|\cdot|_{K_\varepsilon})$ are integrable on $A$, $\varepsilon\in [0,1]$.

Observe that for each $\varepsilon>0$ the function $t\to |t|_{K_\varepsilon}\colon \Rd\to\RR$ is Lipshitz. Indeed, for arbtitrary $t,s\in\Rd$, $
|t|_{K_\varepsilon}= |t-s+s|_{K_\varepsilon}\leq |t-s|_{K_\varepsilon}+|s|_{K_\varepsilon}, 
$
and hence, due to equivalence of finite dimensional norms,
$$
|t|_{K_\varepsilon} - |s|_{K_\varepsilon}\leq |t-s|_{K_\varepsilon}\leq C |t-s|_2
$$
with some $C>0$. Interchanging $t$ and $s$, we obtain the required. By the Rademacher theorem, we obtain that the gradient $\nabla|\cdot|_{K_\varepsilon}$ exists almost everywhere.  For arbitrary $\theta\neq x\in\Rd$ such that $\nabla|x|_{K_\varepsilon}$ exists, set $v = \frac{x}{|x|_2}$. Then
\begin{equation}\label{gradientFromBelow}
    |\nabla|x|_{K_\varepsilon}|_2 \geq (\nabla|x|_{K_\varepsilon}, v)
    =
    \lim_{s\to 0}\frac{|x+s\cdot v|_{K_\varepsilon} - |x|_{K_\varepsilon}}{s}\geq \frac{1}{R}.
\end{equation}
  Using the coarea formula and~\eqref{gradientFromBelow}, we obtain
 \begin{multline}\label{intW|t|Estimate}
\int\limits_{A}w(|t|_{K_\varepsilon}) dt
\leq 
R\int\limits_{A}w(|t|_{K_\varepsilon})|\nabla|t|_{K_\varepsilon}|_2 dt
\\\leq 
R\int_{h}^H\int_{|s|_{K_\varepsilon}=\rho}w(|s|_{K_\varepsilon})dsd\rho
=
R\cdot \mes \partial K_\varepsilon\int_h^H w(\rho)\rho^{d-1}d\rho,
\end{multline}
which implies integrability of $w(|\cdot|_{K_\varepsilon})$. Due to absolute continuity of the integral, for each $\gamma>0$ there exists $\nu_0(\gamma)>0$ such that $\int_B w(|t|_{K_0})dt < \gamma$ provided $\mes B\leq \nu_0(\gamma)$.

Since $K_0\subset K_\varepsilon$ for all $\varepsilon\in [0,1]$ and $K_\varepsilon\in \setsClass$, there exists $\beta\in (0,1]$ such that $\beta|t|_{K_0}\leq |t|_{K_\varepsilon}\leq  |t|_{K_0}$ for all $t\in \Rd$. 
For each $\varepsilon\in [0,1]$ and $\rho$ from the unit sphere $S^{d-1}$ of $\Rd$, set $\alpha(\rho,\varepsilon) = \frac{|\rho|_{K_\varepsilon}}{|\rho|_{K_0}}$.  Then $1\geq \alpha(\rho,\varepsilon)\geq  \beta>0$ for all $\rho\in S^{d-1}$ and $\varepsilon\in [0,1]$. 

Let $B\subset A$ be a measurable set. For $\rho\in S^{d-1}$ and $\varepsilon\in [0,1]$, let $l(B;\rho, \varepsilon)$ be the set of values $|t|_{K_\varepsilon}$ for points $t\in B$ that belong to the ray with origin at $\theta$ and direction $\rho$. 

Let $\varepsilon\in [0,1]$. Together with  the set $B$ we consider the set $B_\varepsilon$ such that $l(B_\varepsilon;\rho,\varepsilon) = \alpha(\rho,\varepsilon)l(B;\rho,\varepsilon)$ for all $\rho\in S^{d-1}$. Then $\mes B_\varepsilon \leq \mes B$ and
\begin{gather*}
    \int_B w(|t|_{K_\varepsilon})dt 
    =
    \int_{S^{d-1}} \int_{l(B;\rho,\varepsilon)}s^{d-1} w(s)dsd\rho
    \\=
    \int_{S^{d-1}} \int_{\alpha(\rho,\varepsilon)l(B;\rho,\varepsilon)}\frac{u^{d-1}}{\alpha^{d}(\rho,\varepsilon)} w\left(\frac{u}{\alpha(\rho,\varepsilon)}\right)dud\rho
    \\\leq 
    \frac{1}{\beta^d}\int_{S^{d-1}} \int_{l(B_\varepsilon;\rho,\varepsilon)}u^{d-1} w\left(\frac{u}{\alpha(\rho,\varepsilon)}\right)dud\rho
    =
    \frac{1}{\beta^d}\int_{B_\varepsilon}w(|t|_{K_0})dt.
\end{gather*}
Hence for each $\gamma>0$ and $\varepsilon \in [0,1]$,  one has
$$
   \int_B w(|t|_{K_\varepsilon})dt \leq  \frac{1}{\beta^d}\int_{B_\varepsilon}w(|t|_{K_0})
   \leq \gamma,
$$
provided $\mes B <\nu_0(\gamma\cdot \beta^d)$.
\end{proof}

\subsection{Proof of Lemma~\ref{l::fourthSummand}}\label{app::fourthSummandProof}
\begin{proof}
For convenience we set $K_0 = K$. Since  $w\in \weightedLp_1(h,+\infty)$, we have 
\begin{equation}\label{tailIsSmall}
    \int_H^\infty w(r)r^{d-1}dr = o(1)\text{ as } H\to\infty.
\end{equation}
For all $H>h$,
\begin{gather*}
\int\limits_{\Rd\setminus hK_\varepsilon}|w(|t|_K)-
w(|t|_{K_\varepsilon})|dt 
\\=
\int\limits_{HK_\varepsilon\setminus hK_\varepsilon}|w(|t|_K)-
w(|t|_{K_\varepsilon})|dt 
 +
\int\limits_{\Rd\setminus HK_\varepsilon}|w(|t|_K)-
w(|t|_{K_\varepsilon})|dt 
\\ \leq 
\int\limits_{HK_\varepsilon\setminus hK_\varepsilon}|w(|t|_K)-
w(|t|_{K_\varepsilon})|dt 
 +
\int\limits_{\Rd\setminus HK}w(|t|_K)dt 
 +
\int\limits_{\Rd\setminus HK_\varepsilon}w(|t|_{K_\varepsilon})dt.
\end{gather*}
Using~\eqref{tailIsSmall} and the same arguments as in the proof of~\eqref{intW|t|Estimate}, one can prove that the last two summands tend to $0$ as $H\to \infty$ uniformly on $\varepsilon$. Hence for arbitrary $\gamma>0$ one can choose $H$ such that the sum of the two last summands is less than $\gamma$ for all  $\varepsilon\in [0,1]$. 

Next, having the chosen above $H$ fixed, we estimate the  first summand. 

Using the Luzin theorem, we find a continuous on $A:=\left[h,\sup_{t\in HK_1\setminus hK}|t|_K\right]$ function $\overline{w}$ that is different from $w$ on a set of points $E\subset A$ such that
$$
\mes  \partial K_1\int_{E}s^{d-1} ds < \nu,
$$
where, using Lemma~\ref{l::intw(|t|)Estimate}, $\nu>0$ is chosen in a such a way that 
$
\int\limits_{B}w(|t|_{K_\varepsilon})dt < \gamma
$
 for all $\varepsilon\in [0,1]$ and for each measurable $B\subset HK_1\setminus hK$, provided $\mes B<2\nu$. For $\varepsilon\in [0,1]$ set
$
E_\varepsilon:=\{t\in HK_1\setminus hK\colon |t|_{K_\varepsilon}\in E\}.
$
Since the sets $K_\varepsilon$ are embedded and convex, $\mes \partial K_\varepsilon$ is a non-decreasing function of $\varepsilon\in [0,1]$,  see e.g.\cite[\S7]{bonnesen}. Hence for each $\varepsilon \in [0,1]$,
\begin{equation}\label{mesE}
 \mes  E_\varepsilon 
= 
\int_{E}\mes  \partial (sK_\varepsilon)ds 
=
\int_{E}s^{d-1}\mes  \partial K_\varepsilon ds
\leq 
\mes  \partial K_1\int_{E}s^{d-1} ds < \nu.   
\end{equation}

Then for each $\varepsilon\in (0,1]$, one has
\begin{gather*}
    \int\limits_{HK_\varepsilon\setminus hK_\varepsilon}|w(|t|_K)-
w(|t|_{K_\varepsilon})|dt 
\\\leq 
    \int\limits_{(HK_\varepsilon\setminus hK_\varepsilon)\setminus(E_0\cup E_\varepsilon)}|w(|t|_K)-
w(|t|_{K_\varepsilon})|dt 
+
    \int\limits_{E_0\cup E_\varepsilon}|w(|t|_K)-
w(|t|_{K_\varepsilon})|dt 
\\\leq 
 \int\limits_{HK_1\setminus hK}|\overline{w}(|t|_K)-
\overline{w}(|t|_{K_\varepsilon})|dt 
+
 \int\limits_{E_0\cup E_\varepsilon}w(|t|_K)dt 
 +
 \int\limits_{E_0\cup E_\varepsilon}w(|t|_{K_\varepsilon})dt.
\end{gather*}
By the choice of $\nu$ and~\eqref{mesE}, we obtain that 
$$
 \int\limits_{E_0\cup E_\varepsilon}w(|t|_K)dt 
 +
 \int\limits_{E_0\cup E_\varepsilon}w(|t|_{K_\varepsilon})dt\leq 2\gamma
$$
for all $\varepsilon\in [0,1]$. Since $\overline{w}$ is continuous on a compact set $A$, it is uniformly continuous on it. Hence there exists $\delta>0$ such that $|\overline{w}(t_1)- \overline{w}(t_2)|<\frac{\gamma}{\mes HK_1\setminus hK},$ provided $|t_1-t_2|_2<\delta$. For all small $\varepsilon>0$ one has $||t|_K - |t|_{K_\varepsilon}|< \delta$ for all $t\in HK_1\setminus hK$, hence for such $\varepsilon$ we obtain

$$
\int\limits_{HK_\varepsilon\setminus hK_\varepsilon}|\overline{w}(|t|_K)-
\overline{w}(|t|_{K_\varepsilon})|dt < \gamma.
$$
This implies the required.
\end{proof}
\section{Case of homogeneous weight $w$}\label{app::homogenuity}
\subsection{Examples of functions from $\weightsClass_{p'}(0,h)$ and $\weightsClass_{p'}^*(0,h)$}
\begin{lemma}\label{l::powerWIsGood}
Let $p>d$, $h>0$, and $w$ be non-negative and such that $w(t) \cdot t^{d+\alpha}$ is bounded on $(0,h)$ for some $0<\alpha < 1-\frac dp$.
Then $w\in \weightsClass_{p'}(0,h)$. The function $w(t) = \frac {1}{t^{d+\alpha}}$ belongs to $\weightsClass_{p'}^*(0,h)$.
\end{lemma}
\begin{proof}
The following inequalities are valid with some constants $C_1, C_2>0$
\begin{gather*}
\int\limits_{0}^hg_{w,h}^{p'}(u)u^{d-1}du 
=
\int\limits_{0}^h\left(\frac{1}{u^{d-1}}\int\limits_u^hw(t)t^{d-1}dt\right)^{p'}u^{d-1}du
\\ \leq 
C_1\int\limits_{0}^h\left(\frac{1}{u^{d-1}}\int\limits_u^h\frac{1}{t^{1+\alpha}}dt\right)^{p'}u^{d-1}du
 =
C_2\int\limits_{0}^h\left(\frac{u^{-\alpha} - h^{-\alpha}}{u^{d-1}}\right)^{p'}u^{d-1}du
\\ =
C_2\int\limits_{0}^h\left(u^{-\alpha} - h^{-\alpha}\right)^{p'}u^{(1-p')(d-1)}du
 <\infty,
\end{gather*}
since 
$$
(1-p')(d-1)-\alpha p' 
> 
d-1 - p'd+p' -p' +\frac{p'd}{p} 
=-1.
$$
To prove the last statment of the lemma, it is enough to notice, that 
$$
\sup\limits_{s\in [t\nu,t]}|w(t)-w(s)| = \frac{1}{t^{d+\alpha}}\cdot \left(1 - \frac {1}{\nu^{d+\alpha}}\right).$$
\end{proof}
\subsection{Proof of Corollary~\ref{c::multiplicativeKolmogorov}}
\begin{proof}
If $w$ is defined in~\eqref{powerW}, then $w(hs) = h^{\beta}w(s)$ for all $h,s>0$, where $\beta = -(d+\gamma)$. Hence for the function $g_{w,h}$ defined in~\eqref{defG}, one has 
\begin{multline*}
g_{w,h}(h\cdot u) 
=
\frac {1} {(hu)^{d-1}} \int\limits_{hu}^hw(t)t^{d-1}dt 
=
\frac {1} {(hu)^{d-1}} \int\limits_{u}^1w(hs)(hs)^{d-1}d(hs)
\\=
\frac {h^{1+\beta}} {u^{d-1}} \int\limits_{u}^1w(s)s^{d-1}ds
=
h^{1+\beta} g_{w,1}(u) ,u\in [0,1],
\end{multline*}
\begin{multline}\label{constantA}
\|g_{w,h}\|_{\weightedLp_{p'}(0,h)}^{p'} 
=
\int_0^h g_{w,h}^{p'}(u) u^{d-1}du
=
\int_0^1 g_{w,h}^{p'}(hv) (hv)^{d-1}d(hv)
\\=
h^{(1+\beta)p'+d}\int_{0}^1g_{w,1}^{p'}(v)v^{d-1}dv
=
h^{(1+\beta)p'+d}\|g_{w,1}\|_{\weightedLp_{p'}(0,1)}^{p'},
\end{multline}
\begin{equation}\label{constantC}
    \int_0^hg_{w,h}^{p'-1}(u)du 
    =     
    \int_0^1g_{w,h}^{p'-1}(hv)d(hv) 
    = 
    h^{(1+\beta)(p'-1)+1}    \int_0^1g_{w,1}^{p'-1}(v)dv 
\end{equation}
and
\begin{equation}\label{constantB}
\int_{h}^\infty w(\rho) \rho^{d-1}d\rho
=
h^{d+\beta}\int_1^\infty w(r)r^{d-1}dr= \frac{h^{d+\beta}}{\gamma}.
\end{equation}
Hence the inequality from Theorem~\ref{th::additiveKolmogorov} can be rewritten as~\eqref{additiveInequalityForHomW}. 
Set $h$ according to~\eqref{multKolmogorovH} 
and using equalities
$$
1-\left(1+\beta + \frac{d}{p'}\right)\cdot \frac{p}{p-d} 
=
\frac{p-d-\left(p+p\beta+ \frac{p}{p'}d\right)}{p-d} 
=
-\frac{p\beta+pd}{p-d} = \alpha,
$$ 
we obtain
\begin{gather*}
\|D_{K,w} f\|_{L_\infty(C)}
\leq 
 X \||\nabla f|_{\polarK}\|_{L_p(C)}
 \left(\frac{\|f\|_{L_\infty(C)}}{ \||\nabla f|_{\polarK}\|_{L_p(C)}}\cdot \frac{X^{p'-1}}{Z}\right)^{1-\alpha}
 \\ + 
 Y
\|f\|_{L_\infty(C)}  \left(\frac{\|f\|_{L_\infty(C)}}{ \||\nabla f|_{\polarK}\|_{L_p(C)}}\cdot \frac{X^{p'-1}}{Z}\right)^{-\alpha}
\\=
\frac{X^{p'} + YZ}{X^{(p'-1)\alpha}\cdot Z^{1-\alpha}}\|f\|_{L_\infty(C)}^{1-\alpha}\||\nabla f|_{\polarK}\|_{L_p(C)}^\alpha,
\end{gather*}
which proves~\eqref{multInequality}.

Let $K$ be a polytope now. We prove that inequality~\eqref{multInequality} becomes equality for each function $\psi_{K,w, h}$ defined in~\eqref{extremalPsi}.
Using~\eqref{constantC}, we obtain
\begin{equation}\label{psiNorm}
\|\psi_{K,w, h}\|_{L_\infty(C)} = \psi_{K,w, h}(\theta) =  \frac12\int_0^hg_{w,h}^{p'-1}(u)du = 
 h^{(1+\beta)(p'-1)+1} Z.
\end{equation}
From Lemmas~\ref{l::radialFunctionIntegral} and~\ref{l::extremalFunction} and equality~\eqref{constantA} it follows that
\begin{multline}\label{nablaPsiNorm}
\||\nabla \psi_{K,w, h}|_{\polarK}\|_{L_p(C)}
=
\left(\int_{h\KC}g_{w,h}^{(p'-1)p}(|y|_K)dy\right)^{1/p}
\\=
\left(\int_{h\KC}g_{w,h}^{p'}(|y|_K)dy\right)^{1/p}
=
\left(d\cdot \mes(\KC)\|g_{w,h}\|_{\weightedLp_{p'(0,h)}}^{p'}\right)^{1/p}
\\=
\left(d\cdot  \mes(\KC) h^{(1+\beta)p'+d}\|g_{w,1}\|_{\weightedLp_{p'(0,1)}}^{p'}\right)^{1/p}
=
X^{\frac{p'}{p}}h^{\frac{(1+\beta)p'+d}{p}}
=
X^{p'-1}h^{\frac{(1+\beta)p'+d}{p}}.
\end{multline}
Since the function $\psi_{K,w, h}$ is extremal in Theorem~\ref{th::additiveKolmogorov}, using~\eqref{constantB} we obtain
\begin{multline}\label{DpsiNorm}
\|D_{K,w} \psi_{K,w, h}\|_{L_\infty(C)}
=
  X \||\nabla \psi_{K,w, h}|_{\polarK}\|_{L_p(C)} h^{1+\beta + \frac{d}{p'}} + Y
\|\psi_{K,w, h}\|_{L_\infty(C)}h^{d+\beta}
\\=
X^{p'}h^{1+\beta + \frac{d}{p'}+\frac{(1+\beta)p'+d}{p}} + YZh^{d+\beta+(1+\beta)(p'-1)+1} = (X^{p'} + YZ)h^{d+(1+\beta)p'}.
\end{multline}
Using the latter three equalities and direct computations, we obtains that~\eqref{multInequality} becomes equality for the function $\psi_{K,w, h}$.

Sharpness of inequality~\eqref{multInequality} with arbitrary $K\in\setsClass$ can be obtained using approximation of $K$ by polytopes. 
\end{proof}

\section{Proof of Corollary~\ref{c::modulusOfContinuity}}\label{app::modulusOfContinuityProof}
\begin{proof}
Assume that $K$ is a polytope.
First, we prove the statement about the quantity $\Omega$. The fact, that $\Omega(\delta_h)$ and $\Omega(\delta)$ on left-hand sides of~\eqref{additiveModulusOfContinuity} and~\eqref{multModulusOfContinuity} do not exceed corresponding right-hand sides follows from Theorem~\ref{th::additiveKolmogorov} and Corollary~\ref{c::multiplicativeKolmogorov} respectively.

 For each $h>0$, the function
$$
\overline{\psi}_{K,w,h} = \frac{\psi_{K,w,h}}{\| |\nabla \psi_{K,w,h}|_{\polarK}\|_{L_p(h\KC)}}
$$
is extremal in inequality~\eqref{Kolm-add-th}, where $\psi_{K,w,h}$ is defined in~\eqref{extremalPsi}. Moreover, 
$$\| |\nabla \overline{\psi}_{K,w,h}|_{\polarK}\|_{L_p(C)} = \| |\nabla \overline{\psi}_{K,w,h}|_{\polarK}\|_{L_p(h\KC)} =  1,$$
and using the first three equalities in~\eqref{nablaPsiNorm} we obtain
$\|\overline{\psi}_{K,w,h}\|_{L_\infty(C)} 
=\delta_h,
$
which implies the last equality in~\eqref{additiveModulusOfContinuity}.

Let $w$ be the power function~\eqref{powerW} and $\beta=-(d+\gamma)$. Then the exponents on the right-hand sides of~\eqref{psiNorm} and~\eqref{DpsiNorm} are distinct, and hence for each $\delta>0$ one can find $\nu, h>0$ such that for the function $f = \nu \psi_{K,w,h}$ one has $\|f\|_{L_\infty(C)} = \delta$ and $\||\nabla f|_{\polarK}\|_{L_p(C)} = 1$. Since each such function $f$ is extremal in~\eqref{multInequality}, this proves  the last equality in~\eqref{multModulusOfContinuity}.

Denote by $\mathfrak{0}$ the identical zero function. For arbitrary $\delta>0$ and $T\in \mathcal{O}$, 
\begin{gather*}
\sup_{\substack{f\in  W_{\infty,p}^K,\, g\in L_\infty(C)\\ \|f-g\|_{L_\infty(C)}\le \delta}}\|D_{K,w}f-Tg\|_{L_\infty(C)}
\geq 
\sup_{\substack{f\in  W_{\infty,p}^K,\|f\|_{L_\infty(C)}\le \delta}}\|D_{K,w}f-T\mathfrak{0}\|_{L_\infty(C)}
\\ =
\sup_{\substack{f\in  W_{\infty,p}^K,\|f\|_{L_\infty(C)}\le \delta}}\max\left\{\|D_{K,w}f-T\mathfrak{0}\|_{L_\infty(C)}, \|D_{K,w}(-f)-T\mathfrak{0}\|_{L_\infty(C)}\right\}
\\\geq 
\frac 12\sup_{\substack{f\in  W_{\infty,p}^K,\|f\|_{L_\infty(C)}\le \delta}}\left(\|D_{K,w}f-T\mathfrak{0}\|_{L_\infty(C)}+ \|D_{K,w}f+T\mathfrak{0}\|_{L_\infty(C)}\right)
\\\geq 
\sup_{\substack{f\in  W_{\infty,p}^K,\|f\|_{L_\infty(C)}\le \delta}}\|D_{K,w}f\|_{L_\infty(C)} = \Omega(\delta),
\end{gather*}
and hence  
\begin{equation}\label{e>omega}
    \mathcal{E}_{\delta} (\mathcal{O}) \geq \Omega(\delta).
\end{equation}

On the other hand, for each $h,\delta>0$,
\begin{multline}\label{e<omega}
\mathcal{E}_{\delta} (\mathcal{O})\leq \mathcal{E}_{\delta} (\mathcal{L})\leq 
\sup_{f\in  W_{\infty,p}^K}\|D_{K,w}f-D_{K,w,h}f\|_{L_\infty(C)} + \|D_{K,w,h}\|\cdot \delta
\\ =
\left(d\cdot\mes(\KC)\right)^{\frac 1{p'}} \left\|g_{w,h} \right\|_{\weightedLp_{p'}(0,h)}
+2d \cdot \mes(\KC) \left(\int\limits_h^\infty w(\rho)\rho^{d-1}\, d\rho\right)\delta,
\end{multline}
which together with~\eqref{e>omega} finishes the proof of~\eqref{additiveModulusOfContinuity}. 

In the case of $w$ defined by~\eqref{powerW}, setting in~\eqref{e<omega}
$$
h = \left(\delta\cdot \frac{X^{p'-1}}{Z}\right)^{\frac{p}{p-d}},
$$
and applying the same computations as in the proof of Corollary~\ref{c::multiplicativeKolmogorov}, we obtain $\mathcal{E}_{\delta} (\mathcal{O})\leq \Omega(\delta)$, which together with~\eqref{e>omega} finishes the proof of~\eqref{multModulusOfContinuity}.

The  case of arbitrary $K\in\setsClass$ can be proved using approximation of $K$ by polytopes.
\end{proof}
\bibliographystyle{elsarticle-num}
\bibliography{mybibfile}

\begin{thebibliography}{10}
\expandafter\ifx\csname url\endcsname\relax
  \def\url#1{\texttt{#1}}\fi
\expandafter\ifx\csname urlprefix\endcsname\relax\def\urlprefix{URL }\fi
\expandafter\ifx\csname href\endcsname\relax
  \def\href#1#2{#2} \def\path#1{#1}\fi

\bibitem{Stechkin1967}
S.~B. Stechkin, Best approximation of linear operators, Math Notes 1~(2) (1967)
  91--99, (in Russian).

\bibitem{Arestov1996}
V.~V. Arestov, Approximation on unbounded operators by the bounded ones and
  relative extremal problems, Uspehi Mat. Nauk 51~(6) (1996) 88--124.
\newblock \href {https://doi.org/10.1070/RM1996v051n06ABEH003001}
  {\path{doi:10.1070/RM1996v051n06ABEH003001}}.

\bibitem{Landau}
E.~Landau, {\"U}ber einen satz des herrn {L}ittlewood, Rendiconti del Circolo
  Matematico di Palermo (1884-1940) 35 (1913) 265--276.

\bibitem{Kolmogorov1939}
A.~N. Kolmogorov, On inequalities between the upper bounds of the successive
  derivatives of an arbitrary function on an infinite interval, Uchenye Zapiski
  MGU. Math 30~(3) (1939) 3--13, (in Russian).

\bibitem{BKKP}
V.~F. Babenko, N.~P. Korneichuk, V.~A. Kofanov, S.~A. Pichugov, Inequalities
  for derivatives and their applications., Naukova Dumka, Kiev, 2003, (in
  Russian).

\bibitem{Konovalov1978}
V.~N. Konovalov, Exact inequalities for norms of the functions, third partial
  and second mixed derivatives, Mat. Zametki 23~(1) (1978) 67--78, (in
  Russian).

\bibitem{BuslaevTikhomirov1979}
A.~P. Buslaev, V.~M. Tikhomirov, On inequalities for derivatives in
  multivariate case, Mat. Zametki 25~(1) (1979) 54--74, (in Russian).

\bibitem{Timoshin1995}
O.~A. Timoshin, Sharp inequalities between norms of partial derivatives of
  second and third order, Doclady RAN 344 (1995) 20--22, (in Russian).
\newblock \href
  {https://doi.org/https://mathscinet.ams.org/mathscinet-getitem?mr=1361022}
  {\path{doi:https://mathscinet.ams.org/mathscinet-getitem?mr=1361022}}.

\bibitem{Timofeev1985}
V.~G. Timofeev, Landau type inequalities for functions of several variables,
  Math Notes 37~(5) (1985) 676--689, (in Russian).

\bibitem{BKP1997}
V.~F. Babenko, N.~A. Kofanov, S.~A. Pichugov, Multivariate inequalities of
  {K}olmogorov type and their applications, In 'Multivariate Approximation and
  Splines', G. Nerberger, J.W. Schmidt, G. Walz (eds), Birkhuser, Basel (1997)
  1--12.

\bibitem{Babenko2000}
V.~F. Babenko, On sharp {K}olmogorov type inequalities for bivariate functions,
  Dopovidi NAN Ukrainy 5 (2000) 7--11, (in Russian).

\bibitem{Samko}
S.~Samko, A.~Kilbas, O.~I. Marichev, Fractional Integrals and Derivatives:
  Theory and Applications, Gordon and Breach Science Publishers, Yveron, 1993.

\bibitem{Geisberg1965}
S.~P. Geisberg, Generalization of {H}adamard inequality, Sb.Nauchn. Tr.
  Leningr. Mech. Inst. 50 (1965) 42--54.

\bibitem{Arestov1979}
V.~V. Arestov, Inequalities for fractional derivatives on the half-line,
  Approximation Theory. Banach Center Publication, PWN, Warsaw (1979) 19--34.

\bibitem{MagTikh1981}
G.~G. Magaril-Il'jaev, V.~M. Tihomirov, On the {K}olmogorov inequality for
  fractional derivatives on the half-line, Analysis Mathematica 1~(7) (1981)
  37--47.

\bibitem{BCh2001}
V.~F. Babenko, M.~S. Churilova, On inequalities of {K}olmogorov type for
  derivatives of fractional order, Bull Dnepropetrovsk Univ. Math 6 (2001)
  16--20, (in Russian).

\bibitem{BabenkoPichugov2007}
V.~F. Babenko, S.~A. Pichugov, Kolmogorov type inequalities for fractional
  derivatives of {H}{\"o}lder functions of two variables, East Jornal on
  Approximations 3~(13) (2007) 321--329.

\bibitem{Babenko07}
V.~F. Babenko, M.~S. Churilova, Kolmogorov type inequalities for hypersingular
  integrals with homogeneous characteristic, Banach J. Math. Anal. 1~(1) (2007)
  66 -- 77.
\newblock \href {https://doi.org/10.15352/bjma/1240321556}
  {\path{doi:10.15352/bjma/1240321556}}.

\bibitem{BabenkoPichugov2010}
V.~F. Babenko, S.~A. Pichugov, Exact estimates of norms of fractional
  derivatives of multivariate functions sutisfying {H}{\"o}lder conditions,
  Math Notes 87 (2010) 26--34.
\newblock \href {https://doi.org/10.1134/S0001434610010049}
  {\path{doi:10.1134/S0001434610010049}}.

\bibitem{babenko2010}
V.~F. Babenko, D.~A. Levchenko, Kolmogorov type inequalities for hypersingular
  integrals with sign-alternating characteristic, Researches in Mathematics 15
  (2010) 18--27, (in Russian).

\bibitem{BPP2010}
V.~F. Babenko, N.~V. Parfinovich, S.~A. Pichugov, Sharp {K}olmogorov-type
  inequalities for norms of fractional derivatives of multivariate functions,
  Ukr Math J 62 (2010) 343--357.
\newblock \href {https://doi.org/10.1007/s11253-010-0358-y}
  {\path{doi:10.1007/s11253-010-0358-y}}.

\bibitem{BP2012-nabla}
V.~F. Babenko, N.~V. Parfinovich, Kolmogorov type inequalities for norms of
  {R}iesz derivatives of multivariate functions and some applications, Proc.
  Steklov Inst. Math 277 (2012) 9--20.
\newblock \href {https://doi.org/10.1134/S0081543812050033}
  {\path{doi:10.1134/S0081543812050033}}.

\bibitem{BP2012-delta}
V.~F. Babenko, N.~V. Parfinovich, Inequalities of the {K}olmogorov type for
  norms of {R}iesz derivatives of multivariate functions and some of their
  applications, J Math Sci 187 (2012) 9--21.
\newblock \href {https://doi.org/10.1007/s10958-012-1045-3}
  {\path{doi:10.1007/s10958-012-1045-3}}.

\bibitem{BPP2014-delta}
V.~F. Babenko, N.~V. Parfinovich, S.~A. Pichugov, Kolmogorov-type inequalities
  for norms of {R}iesz derivatives of functions of several variables with
  {L}aplacian bounded in ${L}_\infty$ and related problems, Math Notes 95
  (2014) 3--14.
\newblock \href {https://doi.org/10.1134/S0001434614010015}
  {\path{doi:10.1134/S0001434614010015}}.

\bibitem{BP2016-delta}
V.~F. Babenko, N.~V. Parfinovich, Estimation of the uniform norm of
  one-dimensional {R}iesz potential of the partial derivative of a function
  with bounded {L}aplacian, Ukr Math J 68 (2016) 987--999.
\newblock \href {https://doi.org/10.1007/s11253-016-1272-8}
  {\path{doi:10.1007/s11253-016-1272-8}}.

\bibitem{MBDK}
V.~P. Motornyi, V.~F. Babenko, A.~A. Dovgoshei, O.~I. Kusnetsova, Approximation
  theory and harmonic analysis., Naukova Dumka, Kiev, 2012, (in Russian).

\bibitem{Ostrowski1938}
A.~Ostrowski, {\"U}ber die absolut abweichung einer differentienbaren
  funktionen von ihren integralmittelwert, Comment. Math. Hel 10 (1938)
  226--227.
\newblock \href {https://doi.org/10.1134/S0001434610010049}
  {\path{doi:10.1134/S0001434610010049}}.

\bibitem{Dragomir17}
S.~S. Dragomir, Ostrowski type inequalities for {L}ebesgue integral: a survey
  of recent results, Australian J. Math. Anal. Appl. 14~(1) (2017) 1--–287.

\bibitem{Babenko20}
V.~Babenko, Y.~Babenko, O.~Kovalenko, On multivariate {O}strowski type
  inequalities and their applications, Math Ineq Appl 23~(2) (2020) 569--583.
\newblock \href {https://doi.org/10.7153/mia-2020-23-47}
  {\path{doi:10.7153/mia-2020-23-47}}.

\bibitem{adams2003}
R.~Adams, J.~Fournier, Sobolev Spaces, Elsevier Science, 2003.

\bibitem{lieb2001}
E.~Lieb, M.~Loss, Analysis, Crm Proceedings \& Lecture Notes, American
  Mathematical Society, 2001.

\bibitem{federer1969}
H.~Federer, Geometric measure theory, Grundlehren der mathematischen
  Wissenschaften, Springer, 1969.

\bibitem{bonnesen}
T.~Bonnesen, W.~Fenchel, L.~Boron, C.~Christenson, B.~Smith, Theory of Convex
  Bodies, BCS Associates, 1987.

\end{thebibliography}

\end{document}